\documentclass{amsart}[12pt]
\usepackage{amssymb}
\usepackage{amsthm}

\usepackage{hyperref}

 \usepackage[usenames]{color}

\newtheorem{thm}{Theorem}[section]

\newtheorem{question}[thm]{Question}
\newtheorem{definition}[thm]{Definition}
\newtheorem{lemma}[thm]{Lemma}
\newtheorem{proposition}[thm]{Proposition}
\newtheorem{remark}[thm]{Remark}
\newtheorem{algorithm}[thm]{Algorithm}

\newtheorem*{Quest}{Question}

\usepackage[all]{xy}

\newcommand{\A}{\mathbb{A}}
\newcommand{\C}{\mathbb{C}}
\newcommand{\F}{\mathbb{F}}
\newcommand{\Q}{\mathbb{Q}}
\newcommand{\Z}{\mathbb{Z}}

\newcommand{\Qp}{\mathbb{Q}_p}
\newcommand{\Zp}{\mathbb{Z}_p}

\newcommand{\cD}{\mathcal{D}}
\newcommand{\cO}{\mathcal{O}}
\newcommand{\cU}{\mathcal{U}}

\newcommand{\dd}{\mathbf{d}_h}

\newcommand{\fa}{\mathfrak{a}}

\newcommand{\Gal}{\mathrm{Gal}}
\newcommand{\lcm}{\mathrm{lcm}}
\newcommand{\ord}{\mathrm{ord}}
\newcommand{\rank}{\mathrm{rank}}

\DeclareMathSymbol{\mlq}{\mathord}{operators}{``}
\DeclareMathSymbol{\mrq}{\mathord}{operators}{`'}

\title[Shadow lines in the arithmetic of elliptic curves]{Shadow lines in the arithmetic of elliptic curves}

\author{J. S. Balakrishnan}
\address{J. S. Balakrishnan, Mathematical Institute, University of Oxford, Woodstock Road, Oxford OX2 6GG, UK}
\email{balakrishnan@maths.ox.ac.uk}

\author{M. \c{C}iperiani}
\address{M. \c{C}iperiani, Department of Mathematics, The University of Texas at Austin, 1 University Station, C1200 
Austin, Texas 78712, USA}
\email{mirela@math.utexas.edu}

\author{J. Lang}
\address{J. Lang, UCLA Mathematics Department, Box 951555, Los Angeles, CA 90095-1555, USA}
\email{jaclynlang@math.ucla.edu}

\author{B. Mirza}
\address{B. Mirza, Department of Mathematics and Statistics, McGill University, 805 Sherbrooke Street West, Montreal, Quebec, Canada, H3A 0B9}
\email{mirza@math.mcgill.ca}

\author{R. Newton}
\address{R. Newton, Max Planck Institute for Mathematics, Vivatsgasse 7, 53111 Bonn, Germany}
\email{rachel@mpim-bonn.mpg.de}

\date{\today}

\begin{document}

\begin{abstract}
Let $E/\Q$ be an elliptic curve and $p$ a rational prime of good ordinary reduction. For every imaginary quadratic field $K/\Q$ satisfying the Heegner hypothesis for $E$ we have a corresponding line in $E(K)\otimes \Qp$, known as a shadow line. When $E/\Q$ has analytic rank $2$ and $E/K$ has analytic rank $3$, shadow lines are expected to lie in $E(\Q)\otimes \Qp$. If, in addition, $p$ splits in $K/\Q$, then shadow lines can be determined using the anticyclotomic $p$-adic height pairing. We develop an algorithm to compute anticyclotomic $p$-adic heights which we then use to provide an algorithm to compute shadow lines. We conclude by illustrating these algorithms in a collection of examples.
\end{abstract}

\thanks{The authors are grateful to the organizers of the conference ``WIN3: Women in Numbers 3'' for facilitating this collaboration\\
\indent and acknowledge the hospitality and support provided by the Banff International Research Station. \\
\indent During the preparation of this manuscript: the second author was partially supported by NSA grant H98230-12-1-0208 and \\
\indent NSF grant DMS-1352598; the third author was partially supported by NSF grant DGE-1144087.}

\maketitle


\section*{Introduction} Fix an elliptic curve $E/\Q$ of analytic rank $2$ and an odd prime $p$ of good ordinary reduction. Assume that the $p$-primary part of the Tate-Shafarevich group of $E/\Q$ is finite.  Let $K$ be an imaginary quadratic field such that the analytic rank of $E/K$ is $3$ and the Heegner hypothesis holds for $E$, i.e., all primes dividing the conductor of $E/\Q$ split in $K$. We are interested in computing the subspace of $E(K) \otimes \Qp$ generated by the anticyclotomic universal norms.  To define this space, let $K_\infty$ be the anticyclotomic $\Zp$-extension of $K$ and $K_n$ denote the subfield of $K_\infty$ whose Galois group over $K$ is isomorphic to $\Z/p^n\Z$.  The module of \textit{universal norms} is defined by
\[
\mathcal{U} = \bigcap_{n \geq 0} N_{K_n/K}(E(K_n) \otimes \Zp),
\]
where $N_{K_n/K}$ is the norm map induced by the map $E(K_n) \to E(K)$ given by $P \mapsto \underset{\sigma\in \mathrm{Gal}(K_n/K)}\sum P^\sigma$.

Consider 
\[
L_K := \cU \otimes \Qp \subseteq E(K) \otimes \Qp.
\]
 
 By work of Cornut \cite{Cornut}, and Vatsal \cite{Vatsal} our assumptions on the analytic ranks of $E/\Q$ and $E/K$ together with the assumed finiteness of the $p$-primary part of the Tate-Shafarevich group of $E/\Q$ imply that $\dim L_K \geq 1$. Bertolini \cite{Bertolini} showed that $\dim L_K = 1$ under certain conditions on the prime $p$. Wiles and \c Ciperiani \cite{CW}, \cite{Ciperiani} have shown that Bertolini's result is valid whenever $\Gal (\Q(E_p)/\Q)$ is not solvable; here $E_p$ denotes the full $p$-torsion of $E$ and $\Q(E_p)$ is its field of definition.
The $1$-dimensional $\Qp$-vector space $L_K$ is known as the \emph{shadow line} associated to the triple $(E,K, p)$. 

Complex conjugation acts on $E(K) \otimes \Qp$, and we consider its two eigenspaces $E(K)^+\otimes \Qp$ and $E(K)^-\otimes \Qp$.
Observe that $E(K)^+\otimes \Qp =E(\Q) \otimes \Qp$. By work of Skinner-Urban \cite{SU}, Nekov\'a\v r \cite{Nekovar}, Gross-Zagier \cite{GZ}, and Kolyvagin \cite{Kolyvagin} we know that
\[
\dim E(K)^+\otimes \Qp \geq 2 \indent \text{and}\indent \dim E(K)^-\otimes \Qp =1.
\]
Then by the Sign Conjecture \cite{MRgrowth} we expect that
\[
L_K \subseteq E(\Q) \otimes \Qp.
\]

Our main motivating question is the following: 
\begin{Quest}[Mazur, Rubin] 
As $K$ varies, we presumably get different shadow lines $L_K$ -- what are these lines, and how are they distributed in $E(\Q) \otimes \Qp$ ? 
\end{Quest}
 
In order to gather data about this question one can add the assumption that $p$ splits in $K/\mathbb{Q}$ and then make use of the \emph{anticyclotomic $p$-adic height pairing} on $E(K)\otimes\mathbb{Q}_p$. It is known that $\mathcal{U}$ is contained in the kernel of this pairing \cite{MTpairing}. In fact, in our situation we expect that $\mathcal{U}$ equals the kernel of the anticyclotomic $p$-adic height pairing. Indeed we have $\dim E(K)^-\otimes \Qp =1$ and the weak Birch and Swinnerton-Dyer Conjecture for $E/\Q$ predicts that $\dim E(\Q) \otimes \Qp=2$ from which the statement about $\cU$ follows by the properties of the anticyclotomic $p$-adic height pairing and its expected non-triviality. (This is discussed in \S \ref{sec:Shadow} in further detail.) Thus computing the anticyclotomic $p$-adic height pairing allows us to determine the shadow line $L_K$.

Let $\Gamma(K)$ be the Galois group of the maximal $\Zp$-power extension of $K$, and let $I(K) = \Gamma(K) \otimes_{\Zp} \Qp$. Identifying $\Gamma(K)$ with an appropriate quotient of the idele class group of $K$, Mazur, Stein, and Tate \cite[\S 2.6]{MST} gave an explicit description of the universal $p$-adic height pairing
\[
( \, , ) : E(K) \times E(K) \to I(K).
\]
One obtains various $\Qp$-valued height pairings on $E$ by composing this universal pairing with $\Qp$-linear maps $I(K) \to \Qp$.  The kernel of such a (non-zero) $\Qp$-linear map corresponds to a $\Zp$-extension of $K$.

In particular, the anticyclotomic $\Zp$-extension of $K$ corresponds to a $\Qp$-linear map $\rho : I(K) \to \Qp$ such that $\rho \circ c = -\rho$, where $c$ denotes complex conjugation. The resulting anticyclotomic $p$-adic height pairing is denoted by $( \, , )_\rho$. One key step of our work is an explicit description of the map $\rho$, see \S \ref{sec:character}. As in \cite{MST}, for $P \in E(K)$ we define the anticyclotomic $p$-adic height of $P$ to be $h_\rho(P) = -\frac{1}{2}(P, P)_\rho$. Mazur, Stein, and Tate \cite[\S2.9]{MST} provide the following formula\footnote{The formula appearing in \cite[\S2.9] {MST} contains a sign error which is corrected here.} for the anticyclotomic $p$-adic height of a point $P \in E(K)$:
\[
h_\rho(P ) = \rho_\pi(\sigma_\pi(P)) - \rho_\pi(\sigma_\pi(P^c)) + \sum_{w \nmid p\infty} \rho_w(d_w(P )),
\]
where $\pi$ is one of the prime divisors of $p$ in $K$ and the remaining notation is defined in \S \ref{sec:Height}. An algorithm for computing $\sigma_\pi$ was given in \cite{MST}. Using our explicit description of $\rho$, in \S \ref{sec:Height} we find a computationally feasible way of determining the contribution of finite primes $w$ which do not divide $p$. This enables us to compute anticyclotomic $p$-adic height pairings.  

We then proceed with a general discussion of shadow lines and their identification in $E(\Q) \otimes \Qp$, see \S \ref{sec:Shadow}. In \S \ref{Algs} we present the algorithms that we use to compute anticyclotomic $p$-adic heights and shadow lines. We conclude by displaying in \S\ref{Examples} two examples of the computation of shadow lines $L_K$ on the elliptic curve ``389.a1" with the prime $p=5$ and listing the results of several additional shadow line computations.

\section{Anticyclotomic character}\label{sec:character}

Let $K$ be an imaginary quadratic field with ring of integers $\cO_K$ in which $p$ splits as $p \cO_K=\pi\pi^c$, where $c$ denotes complex conjugation on $K$. Let $\A^\times$ be the group of ideles of $K$. We also use $c$ to denote the involution of $\A^\times$ induced by complex conjugation on $K$. For any finite place $v$ of $K$, denote by $K_v$ the completion of $K$ at $v$, $\cO_v$ the ring of integers of $K_v$, and $\mu_v$ the group of roots of unity in $\cO_v$. Let $\Gamma(K)$ be the Galois group of the maximal $\Zp$-power extension of $K$. As in \cite{MST}, we consider the idele class $\Qp$-vector space $I(K) =\Gamma(K) \otimes_{\Zp} \Qp $.  By class field theory $\Gamma(K)$ is a quotient of $J' := \A^\times/\overline{K^\times \C^\times \prod_{w \nmid p} \cO_w^\times}$ by its finite torsion subgroup $T$, see the proof of Theorem 13.4 in \cite{Washington}. The bar in the definition of $J'$ denotes closure in the idelic topology, and the subgroup $T$ is the kernel of the $N$-th power map on $J'$ where $N$ is the order of the finite group 
\[
\A^\times/\overline{K^\times \C^\times \prod_{w \nmid p} \cO_w^\times}(1 + \pi\cO_\pi) (1 + \pi^c\cO_{\pi^c}). 
\]
Thus we have
\begin{equation}\label{I(K)}
I(K) = J'/T \otimes_{\Zp}\Qp.
\end{equation} 
We shall use this idelic description of $\Gamma(K)$ in what follows.

\begin{definition}[Anticyclotomic $p$-adic idele class character] \label{anticyclotomic}
An \emph{anticyclotomic} $p$-adic idele class character is a continuous homomorphism
\[
\rho: \A^\times/K^\times\rightarrow\Zp
\]
such that $\rho\circ c = -\rho$.
\end{definition}

\begin{lemma}\label{factoring characters}
Every $p$-adic idele class character 
\[
\rho: \A^\times/K^\times\rightarrow\Zp
\]
factors via the natural projection
\[
\A^\times/K^\times\twoheadrightarrow \A^\times / \Bigl(K^\times\C^\times\prod_{w\nmid p}{\cO_w^\times}\prod_{v\mid p}{\mu_v}\Bigr).
\]
\end{lemma}

\begin{proof}
This is an immediate consequence of the fact that $\Zp$ is a torsion-free pro-$p$ group.
\end{proof}

The aim of this section is to define a non-trivial anticyclotomic $p$-adic idele class character. By the identification \eqref{I(K)}, such a character will give rise to a $\Qp$-linear map $I(K)\rightarrow \Qp$ which cuts out the anticyclotomic $\Zp$-extension of $K$.

\subsection{The class number one case}
We now explicitly construct an anticyclotomic $p$-adic idele class character $\rho$ in the case when the class number of $K$ is $1$. 

Recall our assumption that $p$ splits in $K/\Q$ as $p\cO_K=\pi\pi^c$ and let 
\[
U_\pi = 1 + \pi\cO_\pi \indent \text{and} \indent U_{\pi^c} = 1 + \pi^c\cO_{\pi^c}. 
\]
Define a continuous homomorphism
\[
\varphi : \A^\times \to U_\pi \times U_{\pi^c}
\]  
as follows.  Let $(x_v)_v \in \A^\times$.  Under our assumption that $K$ has class number $1$, we can find $\alpha \in K^\times$ such that 
\[
\alpha x_v \in \cO_v^\times \indent \text{ for all finite }v. 
\]
Indeed, the ideal $\fa_v$ corresponding to the place $v$ is principal, say generated by $\varpi_v \in \cO_K$.  Then take $\alpha = \prod_{v} \varpi_v^{-\ord_v(x_v)}$, where the product is taken over all finite places $v$ of $K$.  We define
\begin{equation} \label{varphi}
\varphi((x_v)_v) = ((\alpha x_\pi)^{p - 1}, (\alpha x_{\pi^c})^{p - 1}).
\end{equation}
Note that since $p$ is split in $K$ we have $\cO_\pi^\times \cong \Zp^\times \cong \mu_{p - 1} \times U_\pi$, and similarly for $\pi^c$.  To see that $\varphi$ is independent of the choice of $\alpha$, we note that any other choice $\alpha' \in K^\times$ differs from $\alpha$ by an element of $\cO_K^\times$.  Since $K$ is an imaginary quadratic field, $\cO_K^\times$ consists entirely of roots of unity.  In particular, under the embedding $K \hookrightarrow K_\pi$ we see that $\cO_K^\times \hookrightarrow \mu_{p - 1}$.  Thus, any ambiguity about $\alpha$ is killed when we raise $\alpha$ to the $(p - 1)$-power.  Therefore, $\varphi$ is well-defined. The continuity of $\varphi$ is easily verified.

\begin{proposition}
Suppose that $K$ has class number $1$. Then the map $\varphi$ defined in \eqref{varphi} induces an isomorphism of topological groups
\[
\A^\times/\Bigl(K^\times\C^\times\prod_{w\nmid p}{\cO_w^\times}\prod_{v\mid p}{\mu_v}\Bigr)\rightarrow U_\pi \times U_{\pi^c}.
\]
\end{proposition}

\begin{proof}
For $v\in\{\pi,\pi^c\}$, the $p$-adic logarithm gives an isomorphism $U_v\cong 1+p\Zp\rightarrow p\Zp$. Hence, raising to the power $(p-1)$ is an automorphism on $U_v$ for $v\in\{\pi,\pi^c\}$ and consequently $\varphi$ is surjective. It is easy to see that $K^\times \C^\times \prod_{w \nmid p} \cO_w^\times \subset \ker \varphi$.  Since $\mu_v\cong \F_p^\times$ for $v\in\{\pi,\pi^c\}$, we have $\prod_{v\mid p}{\mu_v}\subset \ker\varphi$.
We claim that $\ker \varphi = K^\times\C^\times\prod_{w \nmid p} \cO_w^\times\prod_{v\mid p}{\mu_v}$.  Let $(x_v)_v \in \ker \varphi$ and let $\alpha \in K^\times$ be such that $\alpha x_v \in \cO_v^\times$ for all finite $v$. 
It suffices to show that $(\alpha x_v)_v\in\C^\times\prod_{w \nmid p} \cO_w^\times\prod_{v\mid p}{\mu_v}$. This is clear: since $(x_v)_v\in\ker\varphi$, we have $\alpha x_v\in\mu_v$ for $v\in\{\pi,\pi^c\}$.

Finally, since $\varphi$ is a continuous open map, it follows that $\varphi$ induces the desired homeomorphism.
\end{proof}

By Lemma \ref{factoring characters} we have reduced the problem of constructing an anticyclotomic $p$-adic idele class character to the problem of constructing a character 
\begin{equation}
\chi : U_\pi \times U_{\pi^c}\to \Zp
\end{equation}
satisfying $\chi \circ c = -\chi$. Note that this last condition implies that $\chi(x, y) = \chi(x/y^c, 1)$.  Explicitly:
\begin{equation}\label{anticyc property}
\chi(x, y) = -\chi \circ c(x, y) = -\chi(y^c, x^c) = -\chi(y^c, 1) - \chi(1, x^c) = -\chi(y^c, 1) + \chi(x, 1) = \chi(x/y^c, 1).
\end{equation}
In other words, $\chi$ factors via the surjection
\begin{eqnarray*}
f_\pi: U_\pi \times U_{\pi^c}\twoheadrightarrow U_\pi\\
(x,y)\mapsto x/y^c.
\end{eqnarray*}
Therefore, it is enough to define a character $U_\pi \to \Zp$.  Fixing an isomorphism of valued fields $\psi: K_\pi\to \Qp$ gives an identification $U_\pi\cong 1+p\Zp$. Now, up to scaling, there is only one choice of character, namely $\log_p :  1 + p \Zp \to p\Zp$. We write $\log_p$ for the unique group homomorphism $\log_p:\Qp^\times\to (\Qp,+)$ with $\log_p(p)=0$ extending $\log_p :  1 + p \Zp \to p\Zp$. The extension to $\Zp^\times$ of the map $\log_p$ is explicitly given by
\[
\log_p(u)=\frac{1}{p-1}\log_p(u^{p-1}).
\]

We choose the normalization $\rho = \frac{1}{p(p - 1)}\log_p\circ\psi \circ f_\pi \circ \varphi$. We summarize our construction of the anticyclotomic $p$-adic idele class character $\rho$ in the following proposition.

\begin{proposition}\label{charOne}
Suppose that $K$ has class number 1. Fix a choice of isomorphism $\psi:K_\pi\to \Qp$. Consider the map $\rho:\A^\times/K^\times\rightarrow \Zp$ such that
\begin{eqnarray*}
\rho((x_v)_v) 
=\frac{1}{p}\log_p\circ\psi\left(\frac{\alpha x_\pi}{\alpha^cx_{\pi^c}^c} \right)
\end{eqnarray*}
where $\alpha\in K^\times$ is such that $\alpha x_v \in \cO_v^\times$ for all finite $v$. Then $\rho$ is the unique (up to scaling) non-trivial anticyclotomic $p$-adic idele class character.
\end{proposition}

\begin{proof}
Let $\alpha\in K^\times$ be such that $\alpha x_v \in \cO_v^\times$ for all finite $v$. By our earlier discussion and the definition of the extension of $\log_p$ to $\Zp^\times$, we have
\begin{eqnarray*}
\rho((x_v)_v) &=& \frac{1}{p(p - 1)}\log_p\circ \psi\left(\frac{(\alpha x_\pi)^{p - 1}}{(\alpha^cx_{\pi^c}^c)^{p - 1}} \right)\\
&=&\frac{1}{p}\log_p\circ\psi\left(\frac{\alpha x_\pi}{\alpha^cx_{\pi^c}^c} \right).
\end{eqnarray*}
\end{proof}


\subsection{The general case}

There is a simple generalization of the construction of $\rho$ to the case when the class number of $K$ may be greater than one.  Let $h$ be the class number of $K$. We can no longer define the homomorphism $\varphi$ of \eqref{varphi} on the whole of $\A^\times$ because $\cO_K$ is no longer assumed to be a principal ideal domain. However, we can define 
\[
\varphi_h: (\A^\times)^h\to U_\pi\times U_{\pi^c}
\]
in a similar way, as follows. Let $\fa_v$ be the ideal of $K$ corresponding to the place $v$.  Then $\fa_v^h$ is principal, say generated by $\varpi_v \in \cO_K$.  For $(x_v)_v\in\A^\times$ we set $\alpha(v) = {\varpi_v}^{-\ord_v(x_v)}$.  Then $\alpha(v)x_v^h \in \cO_v^\times$ and $\alpha(v) \in \cO_w^\times$ for all $w \neq v$.  Note that $\alpha(v) = 1$ for all but finitely many $v$.  Set $\alpha = \prod_v \alpha(v)$ and observe that $\alpha x_v^h \in \cO_v^\times$ for all $v$.  Then we define $\varphi_h$ by
\begin{equation}\label{varphi_h}
\varphi_h((x_v)_v^h)=((\alpha x_\pi^h)^{p - 1}, (\alpha x_{\pi^c}^h)^{p - 1}).
\end{equation}
Fix an isomorphism $\psi:K_\pi\to \Qp$. As before, we can now use the $p$-adic logarithm to define an anticyclotomic character $\rho:(\A^\times)^h\to \Zp$ by setting 
\[
\rho = \frac{1}{p(p - 1)}\log_p \circ \psi \circ f_\pi \circ \varphi_h.
\]
We extend the definition of $\rho$ to the whole of $\A^\times$ by setting $\rho((x_v)_v)=\frac{1}{h}\rho((x_v)_v^h)$.  

As in Proposition \ref{charOne}, we now summarize our construction of the anticyclotomic $p$-adic idele class character in this more general setting.
\begin{proposition} \label{general character}
Let $h$ be the class number of $K$. Fix a choice of isomorphism $\psi:K_\pi\to \Qp$. Consider the map $\rho:\A^\times/K^\times\rightarrow \frac{1}{h}\Zp$ such that
\begin{eqnarray*}
\rho((x_v)_v)= \frac{1}{hp}\log_p\circ\psi\left(\frac{\alpha x_\pi^h}{\alpha^cx_{\pi^c}^{ch}} \right)
\end{eqnarray*}
 where $\alpha\in K^\times$ is such that $\alpha x_v^h \in \cO_v^\times$ for all finite $v$. Then $\rho$ is the unique (up to scaling) non-trivial anticyclotomic $p$-adic idele class character.
\end{proposition}

\begin{remark}
Note that $\rho:\A^\times/K^\times\rightarrow \frac{1}{h}\Zp$, so if $p\mid h$ then $\rho$ is not strictly an anticyclotomic idele class character in the sense of Definition \ref{anticyclotomic}. However, the choice of scaling of $\rho$ is of no great importance since our purpose is to use $\rho$ to define an anticyclotomic height pairing on $E(K)$ and compute the kernel of this pairing.
\end{remark}

\begin{remark} \label{alpha}
The ideal $\prod_v \fa_v^{-h\, \ord_v(x_v)}$ is  principal and a generator of this ideal is the element $\alpha \in K$ that we use when evaluating the character $\rho$ defined in Proposition \ref{general character}.
\end{remark}


\section{Anticyclotomic $p$-adic height pairing}\label{sec:Height}

We wish to compute the anticyclotomic $p$-adic height $h_\rho$ using our explicit description of the anticyclotomic idele class character $\rho$ given in Proposition \ref{general character}.  For any finite prime $w$ of $K$, the natural inclusion $K_w^\times\hookrightarrow \A^\times$ induces a map $\iota_w : K_w^\times \rightarrow I(K)$, and we write $\rho_w = \rho \circ \iota_w$. For every finite place $w$ of $K$  and every non-zero point $P \in E(K)$ we can find $d_w(P) \in \cO_w$ and $a_w(P), b_w(P) \in \cO_w$, each relatively prime to $d_w(P)$, such that
\begin{equation}\label{denom}
\left(\iota_w(x(P)), \iota_w(y(P)) \right) = \left(\frac{a_w(P)}{d_w(P)^2}, \frac{b_w(P)}{d_w(P)^3} \right).
\end{equation}
We refer to $d_w(P)$ as a \textit{local denominator} of $P$ at $w$.  The existence of $d_w(P)$ follows from the Weierstrass equation for $E$ and the fact that $\cO_w$ is a principal ideal domain.  Finally, we let $\sigma_\pi$ denote the $\pi$-adic $\sigma$-function of $E$.

Given a non-torsion point $P \in E(K)$ such that
\begin{itemize}
\item $P$ reduces to $0$ modulo primes dividing $p$, and
\item $P$ reduces to the connected component of all special fibers of the Neron model of $E$, 
\end{itemize}
we can compute its anticyclotomic $p$-adic height using the following formula\footnote{ The formula appearing in \cite[\S2.9] {MST} contains a sign error which is corrected here.} \cite[\S2.9] {MST} :
\begin{equation}\label{eq:height}
h_\rho(P) = \rho_\pi(\sigma_\pi(P)) - \rho_\pi(\sigma_\pi(P^c)) + \sum_{w \nmid p\infty} \rho_w(d_w (P )).
\end{equation}

In the following lemmas, we make some observations which simplify the computation of $h_{\rho}(P)$.

\begin{lemma} \label{units die}
Let $w$ be a finite prime such that $w\nmid p$. Let $x_w\in K_w^\times$. Then $\rho_w(x_w)$ only depends on $\ord_w(x_w)$. In particular, if $x_w\in\cO_w^\times$, then $\rho_w(x_w)=0$.
\end{lemma}

\begin{proof}
This follows immediately from Lemma \ref{factoring characters}. Alternatively, note that the auxiliary element $\alpha$ used in the definition of $\rho$ only depends on the valuation of $x_w$. 
\end{proof}

\begin{lemma} \label{char-prime}
Let $w$ be a finite prime of $K$. Then $\rho_{w^c}=-\rho_{w}\circ c$. In particular, if $w=w^c$, then $\rho_w=0$.
\end{lemma}

\begin{proof}
This is an immediate consequence of the relations $\rho \circ c = -\rho$ and $c \circ \iota_{\lambda^c} = \iota_\lambda \circ c$.
\end{proof}

Lemma \ref{char-prime} allows us to write the formula \eqref{eq:height} for the anticyclotomic $p$-adic height as follows:
\begin{equation}\label{simplified height}
h_{\rho}(P) = \rho_\pi\left(\frac{\sigma_\pi(P)}{\sigma_\pi(P^c)} \right) +\sum_{\substack{\ell = \lambda\lambda^c\\ \ell \neq p}} \rho_\lambda\left(\frac{d_\lambda(P)}{d_{\lambda^c}(P)^c} \right).
\end{equation}

\begin{remark} In order to implement an algorithm for calculating the anticyclotomic $p$-adic height $h_\rho$, we must determine a finite set of primes which includes all the split primes $\ell = \lambda \lambda^c \nmid p$ for which $\rho_\lambda\left(\frac{d_\lambda(P)}{d_{\lambda^c}(P)^c} \right) \neq 0$.  Let $k_\lambda$ be the residue field  of $K$ at $\lambda$ and set $\cD (P )= \prod_{\lambda \nmid p\infty} (\#k_\lambda)^{\ord_\lambda(d_\lambda(P ))}.$
It turns out that $\cD (P )$ can be computed easily from the leading coefficient of the minimal polynomial of the $x$-coordinate of $P$ \cite[Proposition 4.2]{BCS}. 
Observe that $\rho_\lambda\left(\frac{d_\lambda(P)}{d_{\lambda^c}(P)^c} \right) \neq 0$ implies that $\ord_\lambda(d_\lambda(P ))\neq 0$ or $\ord_{\lambda^c}(d_{\lambda^c}(P)) \neq 0$. Hence, the only primes $\ell\neq p$ which contribute to the sum in \eqref{simplified height} are those that are split in $K/\Q$ and divide $\cD( P)$. However, in the examples that we have attempted, factoring $\cD( P)$ is difficult due to its size.
\end{remark}

\medskip

We now package together the contribution to the anticyclotomic $p$-adic height coming from primes not dividing $p$. Consider the ideal $x(P ) \cO_K$ and denote by $\delta(P)\subset \cO_K$ its denominator ideal. Observe that by \eqref{denom} we know that all prime factors of $\delta(P )$ appear with even powers. Fix $\dd(P ) \in \cO_K$ as follows: 
\begin{equation} \label{denom-replacement}
\dd(P )\cO_K =\prod_{\mathfrak{q}}{\mathfrak{q}^{h\, \ord_{\mathfrak{q}}(\delta(P))/2}} 
\end{equation}
where $h$ is the class number of $K$, and the product is over all prime ideals $\mathfrak{q}$ in $\cO_K$.

\begin{proposition} \label{packaging}
Let $P \in E(K)$ be a non-torsion point which reduces to $0$ modulo primes dividing $p$, and to the connected component of all special fibers of the Neron model of $E$.
Then the anticyclotomic $p$-adic height of $P$ is
\[
h_\rho(P)= \frac{1}{p}\log_p\left(\psi\left(\frac{\sigma_\pi(P)}{\sigma_\pi(P^c)} \right)\right)+\frac{1}{hp}\log_p\left(\psi\left(\frac{\dd(P )^c}{\dd( P)}\right)\right),
\]
where $\psi:K_\pi\to \Qp$ is the fixed isomorphism.
\end{proposition}

\begin{proof}
By \eqref{eq:height} we have
\begin{equation}\label{eq:height2}
h_\rho(P) = \rho_\pi\left(\frac{\sigma_\pi(P)}{\sigma_\pi(P^c)} \right)  + \sum_{w \nmid p\infty} \rho_w(d_w (P )).
\end{equation}
Let $P=(x,y) \in E(K)$. Since $P$ reduces to the identity modulo $\pi$ and $\pi^c$, we have 
\begin{align*}
\ord_{\pi}(x)=-2e_\pi,& \indent \ord_{\pi}(y)=-3e_\pi, \\ 
\ord_{\pi^c}(x)=-2e_{\pi^c},& \indent \ord_{\pi^c}(y)=-3e_{\pi^c},
\end{align*}
for positive integers $e_\pi$ and $e_{\pi^c}$. Since the $p$-adic $\sigma$ function has the form $\sigma(t)=t+\cdots\in t\Zp[[t]]$, we see that 
\begin{eqnarray*}
\ord_\pi(\sigma_\pi(P))=\ord_\pi\left(\sigma_\pi\left(\frac{-x}{y}\right)\right)=\ord_\pi\left(\frac{-x}{y}\right)=e_\pi
\end{eqnarray*}
and similarly
\[\ord_\pi(\sigma_\pi(P^c))=\ord_\pi\left(\frac{-x^c}{y^c}\right)=\ord_{\pi^c}\left(\frac{-x}{y}\right)=e_{\pi^c}.\]
Thus, 
\begin{equation}
\label{eq:order at pi}
\ord_{\pi}\left(\frac{\sigma_\pi(P)}{\sigma_\pi(P^c)} \right)=e_\pi-e_{\pi^c}.
\end{equation}

Let $\alpha\in K^\times$ generate the principal ideal $\pi^{h}$. By \eqref{eq:order at pi} and the definition of the anticyclotomic $p$-adic idele class character, we have
\begin{eqnarray*}
 \rho_\pi\left(\frac{\sigma_\pi(P)}{\sigma_\pi(P^c)} \right)&=& \frac{1}{hp}\log_p\circ\psi\left(\frac{\alpha^{e_{\pi^c}-e_\pi}\sigma_\pi(P)^h}{(\alpha^c)^{e_{\pi^c}-e_\pi}\sigma_\pi(P^c)^h }\right)\\
&=&\frac{1}{p}\log_p\left(\psi\left(\frac{\sigma_\pi(P)}{\sigma_\pi(P^c)}\right)\right)+\frac{1}{hp}\log_p\left(\psi\left(\frac{\alpha}{\alpha^c}\right)^{e_{\pi^c}-e_\pi}\right).
\end{eqnarray*}

Now it remains to show that
\begin{equation}
\label{eq:cancellation}
\sum_{w \nmid p\infty} \rho_w(d_w (P ))=\frac{1}{hp}\log_p\left(\psi\left(\frac{\dd ( P)^c}{\dd ( P)}\right)\right)-\frac{1}{hp}\log_p\left(\psi\left(\frac{\alpha}{\alpha^c}\right)^{e_{\pi^c}-e_\pi}\right).
\end{equation}
By the definition of $\rho$, we have
\begin{equation} \label{awayfromp}
\sum_{w \nmid p\infty} \rho_w(d_w(P))=\frac{1}{h} \sum_{w \nmid p\infty} \rho_w(d_w(P)^h).
\end{equation}
Since $\ord_w(d_w(P)^h)=\ord_w(\dd ( P))$, Lemma \ref{units die} gives $\rho_w(d_w(P)^h)=\rho_w(\dd ( P))$ for every $w\nmid p\infty$. Substituting this into \eqref{awayfromp} gives
\begin{eqnarray*}
\sum_{w \nmid p\infty} \rho_w(d_w(P))&=&\frac{1}{h} \sum_{w \nmid p\infty}\rho_w(\dd ( P))\\
&=&\frac{1}{h} \sum_{w \nmid p\infty}\rho\circ \iota_w(\dd ( P))\\
&=&\frac{1}{h} \rho\Bigl(\prod_{w \nmid p\infty}\iota_w(\dd ( P))\Bigr).
\end{eqnarray*}
Now $\prod_{w \nmid p\infty}\iota_w(\dd ( P))$ is the idele with entry $\dd ( P)$ at every place $w\nmid p\infty$ and entry $1$ at all other places. Define $\beta\in\cO_K$ by $\dd ( P)=\alpha^{e_\pi}(\alpha^c)^{e_{\pi^c}}\beta$. Thus, by Proposition \ref{general character} and Remark \ref{alpha}, we get
\begin{eqnarray*}
\frac{1}{h} \rho\Bigl(\prod_{w \nmid p\infty}\iota_w(\dd ( P))\Bigr)&=&\frac{1}{hp}\log_p\left(\psi\left(\frac{\beta^c}{\beta} \right)\right)\\
&=&\frac{1}{hp}\log_p\left(\psi\left(\frac{\dd ( P)^c}{\dd ( P)} \right)\right)-\frac{1}{hp}\log_p\left(\psi\left(\frac{\alpha}{\alpha^c}\right)^{e_{\pi^c}-e_\pi}\right)
\end{eqnarray*}
as required. This concludes the proof.
\end{proof}

In \cite{MST}, the authors describe the ``universal'' $p$-adic height pairing $(P,Q)\in I(K)$ of two points $P,Q\in E(K)$. Composition of the universal height pairing with any $\Qp$-linear map $\rho:I(K)\rightarrow \Qp $ gives rise to a canonical symmetric bilinear pairing
\[(\ ,\ )_\rho:E(K)\times E(K)\rightarrow \Qp\]
called the \emph{$\rho$-height pairing}. The $\rho$-height of a point $P\in E(K)$ is defined to be $-\frac{1}{2}(P,P)_\rho$. 

Henceforth, we fix $\rho$ to be the anticyclotomic $p$-adic idele class character defined in \S\ref{sec:character}. The corresponding $\rho$-height pairing is referred to as the \emph{anticyclotomic $p$-adic height pairing}, and it is denoted as follows:
\[
\langle\ , \ \rangle=(\ ,\ )_\rho: E(K)\times E(K)\rightarrow \Qp
\]
Observe that
\[
\langle P,Q\rangle=h_\rho(P)+h_\rho(Q)-h_\rho(P+Q).
\]

Let $E(K)^+$ and $E(K)^-$ denote the $+1$-eigenspace and the $-1$-eigenspace, respectively, for the action of complex conjugation on $E(K)$. Since $\sigma_\pi$ is an odd function, using \eqref{simplified height} we see that the anticyclotomic height satisfies 
\[
h_\rho(P)=0 \indent \text{ for all }P\in E(K)^+\cup E(K)^-.
\] 
Therefore, the anticyclotomic $p$-adic height pairing satisfies 
\begin{equation} \label{restriction vanishes}
\langle E(K)^+, E(K)^+\rangle=\langle E(K)^-,E(K)^-\rangle=0.
\end{equation}
Consequently, if $P\in E(K)^+$ and $Q\in E(K)^-$, then 
\begin{align} \label{pairing-height}
\langle P,Q\rangle & =  h_\rho(P)+h_\rho(Q)-h_\rho(P+Q)\\
&= -\frac{1}{2}\langle P,P\rangle -\frac{1}{2}\langle Q,Q\rangle -h_\rho(P+Q) \nonumber \\
&= -h_\rho(P+Q).\nonumber
\end{align}

\bigskip 

\section{The shadow line} \label{sec:Shadow}

Let $E$ be an elliptic curve defined over $\Q$ and $p$ an odd prime of good ordinary reduction. Fix an imaginary quadratic extension $K/\Q$ satisfying the Heegner hypothesis for $E/\Q$ (i.e., all primes dividing the conductor of $E/\Q$ split in $K$). Consider the anticyclotomic $\Zp$-extension $K_\infty$ of $K$. Let $K_n$ denote the subfield of $K_\infty$ whose Galois group over $K$ is isomorphic to $\mathbb{Z}/p^n\mathbb{Z}$. The module of \textit{universal norms} for this $\Zp$-extension is defined as follows:
\[
\cU := \bigcap_{n \geq 0} N_{K_n/K}(E(K_n) \otimes \Zp) \subseteq E(K)\otimes \Zp,
\]
where $N_{K_n/K}$ is the norm map induced by the map $E(K_n) \to E(K)$ given by $P \mapsto \underset{\sigma\in \mathrm{Gal}(K_n/K)}\sum P^\sigma$.

By work of Cornut \cite{Cornut} and Vatsal \cite{Vatsal} we know that for $n$ large enough, we have a non-torsion Heegner point in $E(K_n)$.  Since $p$ is a prime of good ordinary reduction, the trace down to $K_{n-1}$ of the Heegner points defined over $K_n$ is related to Heegner points defined over $K_{n-1}$, see \cite[\S 2]{BCS} for further details. Due to this relation among Heegner points defined over the different layers of $K_\infty$, if the $p$-primary part of the Tate-Shafarevich group of $E/K$ is finite then these points give rise to non-trivial universal norms. Hence, if the $p$-primary part of the Tate-Shafarevich group of $E/K$ is finite then $\cU$ is non-trivial whenever the Heegner hypothesis holds. By \cite{Bertolini}, \cite{CW}, and \cite{Ciperiani} we know that if $\Gal (\Q(E_p)/\Q)$ is not solvable then $\cU\simeq \Zp$.

Consider 
\[
L_K := \cU \otimes \Qp.
\]
If the $p$-primary part of the Tate-Shafarevich group of $E/K$ is finite then $L_K$ is a line in the vector space $E(K) \otimes \Qp$ known as the \emph{shadow line} associated to the triple $(E, K, p)$. The space $E(K) \otimes \Qp$ splits as the direct sum of two eigenspaces under the action of complex conjugation
\[
E(K) \otimes \Qp= E(K)^+ \otimes \Qp \oplus E(K)^- \otimes \Qp. 
\]
Observe that 
\[
E(K)^+ \otimes \Qp= E(\Q) \otimes \Qp \indent\text{and}\indent E(K)^- \otimes \Qp\simeq E^K(\Q) \otimes \Qp, 
\]
where $E^K$ denotes the quadratic twist of $E$ with respect to $K$. Since the module $\cU$ is fixed by complex conjugation, the shadow line $L_K$ lies in one of the eigenspaces:
\[
L_K \subseteq E(\Q) \otimes \Qp \indent\text{or}\indent L_K \subseteq E(K)^- \otimes \Qp.
\]
The assumption of the Heegner hypothesis forces the analytic rank of $E/K$ to be odd, and hence the dimension of $E(K) \otimes \Qp$ is odd by the Parity Conjecture \cite{Nekovar} and our assumption of the finiteness of the $p$-primary part of the Tate-Shafarevich group of $E/K$. Hence, $\dim E(K)^- \otimes \Qp \neq \dim E(\Q) \otimes \Qp$. The Sign Conjecture states that $L_K$ is expected to lie in the eigenspace of higher dimension \cite{MRgrowth}.  

Our main motivating question is the following: 
\begin{question}[Mazur, Rubin] \label{Quest}
Consider an elliptic curve $E/\Q$ of positive even analytic rank $r$, an imaginary quadratic field $K$ such that $E/K$ has analytic rank $r+1$, and a prime $p$ of good ordinary reduction such that the $p$-primary part of the Tate-Shafarevich group of $E/\Q$ is finite. By the Sign Conjecture, we expect $L_K$ to lie in $E(\Q)\otimes \Qp$. As $K$ varies, we presumably get different shadow lines $L_K$. What are these lines and how are they distributed in $E(\Q)\otimes \Qp$? 
\end{question}
 
Note that in the statement of the above question we make use of the following results: 
\begin{enumerate}
\item  Since $E/\Q$ has positive even analytic rank we know that $\dim E(\Q)\otimes \Qp \geq 2$ by work of Skinner-Urban \cite[Theorem 2]{SU} and work of Nekovar \cite{Nekovar} on the Parity Conjecture. \label{Qrank}
\item Since our assumptions on the analytic ranks of $E/\Q$ and $E/K$ imply that the analytic rank of $E^K/\Q$ is $1$, by work of Gross-Zagier \cite{GZ} and Kolyvagin \cite{Kolyvagin} we know that
\begin{enumerate}
\item $\dim E(K)^- \otimes \Qp=1$; \label{Krank}
\item the $p$-primary part of the Tate-Shafarevich group of $E^K/\Q$ is finite, and hence the finiteness of the $p$-primary part of the Tate-Shafarevich group of $E/K$ follows from the finiteness of the $p$-primary part of the Tate-Shafarevich group of $E/\Q$. \label{Sha}
\end{enumerate}
\end{enumerate} 
Thus by \eqref{Sha} we know that $L_K \subseteq E(K)\otimes \Qp$, while \eqref{Qrank} and \eqref{Krank} are the input to the Sign Conjecture.

It is natural to start the study of Question \ref{Quest} by considering elliptic curves $E/\Q$ of analytic rank $2$. In this case, assuming that
\begin{equation}\label{algRank}
\rank_\Z E(\Q) =2,
\end{equation}
we identify $L_K$ in $E(\Q)\otimes \Qp$ by making use of the anticyclotomic $p$-adic height pairing, viewing it as a pairing on $E(K)\otimes \Zp$. This method forces us to restrict our attention to quadratic fields $K$ where $p$ splits. It is known that $\cU$ is contained in the kernel of the anticyclotomic $p$-adic height pairing \cite[Proposition 4.5.2]{MTpairing}.  In fact, in our situation, the properties of this pairing and \eqref{algRank} together with the fact that $\dim E(K)^- \otimes \Qp=1$ imply that either $\cU$ is the kernel of the pairing or the pairing is trivial. Thus computing the anticyclotomic $p$-adic height pairing allows us to verify the Sign Conjecture and determine the shadow line $L_K$. 

In order to describe the lines $L_K$ for multiple quadratic fields $K$, we fix two independent generators $P_1, P_2$ of $E(\Q)\otimes \Qp$ (with $E$ given by its reduced minimal model) and compute the slope of $L_K\otimes \Qp$ in the corresponding coordinate system. For each quadratic field $K$ we compute a non-torsion point $R \in E(K)^-$ (on the reduced minimal model of $E$). The kernel of the anticyclotomic $p$-adic height pairing on $E(K)\otimes \Zp$ is generated by $aP_1+bP_2$ for $a,b\in\Zp$ such that $\langle aP_1+bP_2, R\rangle=0$. Then by \eqref{pairing-height} the shadow line $L_K\otimes \Qp$ in $E(\Q)\otimes \Qp$ is generated by $h_\rho(P_2+R) P_1-h_\rho(P_1+R) P_2$ and its slope with respect to the coordinate system induced by $\{P_1, P_2\}$ equals   
\[
-h_\rho(P_1+R)/ h_\rho(P_2+R).
\]

 \bigskip


\section{Algorithms} \label{Algs}

Let $E/\Q$ be an elliptic curve of analytic rank $2$; see \cite[Chapter 4]{Bradshaw} for an algorithm that can provably verify the non-triviality of the second derivative of the $L$-function. Our aim is to compute shadow lines on the elliptic curve $E$. In order to do this using the method described in \S \ref{sec:Shadow} we need to 
\begin{itemize}
\item verify that $\rank_\Z E(\Q)=2$, and 
\item compute two $\Z$-independent points $P_1, P_2 \in E(\Q)$. 
\end{itemize}

By work of Kato \cite[Theorem 17.4]{Kato}, computing the $\ell$-adic analytic rank of $E/\Q$ for any prime $\ell$ of good ordinary reduction gives an upper bound on $\rank_\Z E(\Q)$  (see  \cite[Proposition 10.1]{SW}). Using the techniques in \cite[$\S 3$]{SW}, which have been implemented in \texttt{Sage}, one can compute an upper bound on the $\ell$-adic analytic rank using an approximation of the $\ell$-adic $L$-series, thereby obtaining an upper bound on $\rank_\Z E(\Q)$. Since the analytic rank of $E/\Q$ is $2$, barring the failure of standard conjectures we find that $\rank_\Z E(\Q) \leq 2$. Then using work of Cremona \cite[Section 3.5]{Cremona} implemented in \texttt{Sage}, we search for points of bounded height, increasing the height until we find two $\Z$-independent points $P_1, P_2\in E(\Q)$. We have thus computed a basis of $E(\Q)\otimes \Qp$. 

We will now proceed to describe the algorithms that allow us to compute shadow lines on the elliptic curve $E/\Q$.

\bigskip

\begin{algorithm}\label{Kgenerator} Generator of $E(K)^-\otimes \Qp$.\\

\textit{Input}:  
\begin{itemize}
\item an elliptic curve $E/\Q$ (given by its reduced minimal model) of analytic rank $2$,
\item an odd prime $p$ of good ordinary reduction; 
\item an imaginary quadratic field $K$ such that 
\begin{itemize}
\item the analytic rank of $E/K$ equals $3$, and 
\item all rational primes dividing the conductor of $E/\Q$ split in $K$.    
\end{itemize}
\end{itemize}
\textit{Output}:  A generator of $E(K)^-\otimes \Qp$ (given as a point on the reduced minimal model of $E/\Q$).\\
\begin{enumerate}
\item Let $d\in \Z$ such that $K= \Q(\sqrt d)$. Compute a short model of $E^K$, of the form $y^2 = x^3 + ad^2x + bd^3$. 
\item Our assumption on the analytic ranks of $E/\Q$ and $E/K$ implies that the analytic rank of $E^K/\Q$ is $1$. Compute a non-torsion point\footnote{Note that by \cite{GZ} and \cite{Kolyvagin}  the analytic rank of $E^K/\Q$ being $1$ implies that the algebraic rank of $E^K/\Q$ is $1$ and the Tate-Shafarevich group of $E^K/\Q$ is finite. Furthermore, in this case, computing a non-torsion point in $E^K(\Q)$ can be done by choosing an auxiliary imaginary quadratic field $F$ satisfying the Heegner hypothesis for $E^K/\Q$ such that the analytic rank of $E^K/F$ is $1$ and computing the corresponding basic Heegner point in $E^K(F )$.} of $E^K(\Q)$ and denote it $(x_0,y_0)$. Then $(\frac{x_0}{d}, \frac{y_0\sqrt{d}}{d^2})$ is an element of $E(K)$ on the model $y^2 = x^3 + ax + b$.
\item Output the image of $(\frac{x_0}{d}, \frac{y_0\sqrt{d}}{d^2})$ on the reduced minimal model of $E$.
\end{enumerate}
\end{algorithm}

\bigskip

\begin{algorithm} \label{height} Computing the anticyclotomic $p$-adic height associated to $(E,K,p)$.\\
\textit{Input}: 
\begin{itemize}
\item elliptic curve $E/\Q$ (given by its reduced minimal model);
\item an odd prime $p$ of good ordinary reduction; 
\item an imaginary quadratic field $K$ such that $p$ splits in $K/\Q$;
\item a non-torsion point $P\in E(K)$.
\end{itemize}
\textit{Output}: The anticyclotomic $p$-adic height of $P$.\\

\begin{enumerate}

\item Let $p\cO_K= \pi \pi^c$. Fix an identification $\psi: K_{\pi} \simeq \Qp$. In particular, $v_p(\psi(\pi)) = 1$. 

\item Let $m_0 = \lcm \{c_{\ell}\}$, where $\ell$ runs through the primes of bad reduction for $E/\Q$ and $c_{\ell}$ is the Tamagawa number at $\ell$. Compute\footnote{Note that Step \ref{bad-red} and Step \ref{p-reduction} are needed to ensure that the point whose anticyclotomic $p$-adic height we will compute using formula \eqref{eq:height} satisfies the required conditions. } $R = m_0 P$. \label{bad-red}

\item Determine the smallest positive integer $n$ such that $nR$ and $nR^c$ reduce to $0\in E(\F_p)$ modulo $\pi$. Note that $n$ is a divisor of $\#E(\F_p)$.  Compute\ $T= nR.$ \label{p-reduction}

\item Compute $\dd( R) \in \cO_K$ defined in \eqref{denom-replacement} as a generator of the ideal 
\[
\prod_{\mathfrak{q}}{\mathfrak{q}^{h\, \ord_{\mathfrak{q}}(\delta( R))/2}} 
\]
where $h$ is the class number of $K$, the product is over all prime ideals $\mathfrak{q}$ of $\cO_K$, and $\delta ( R)$ is the denominator ideal of $x(R )\cO_K$.

\item Let $f_n$ denote the $n$th division polynomial associated to $E$. Compute $\dd(T) = \dd(nR) = f_n(R )^h \dd(R )^{n^2}$.  
Note that by Step \eqref{bad-red} and Proposition 1 of Wuthrich \cite{Wuthrich} we see that $f_n(R )^h \dd(R )^{n^2} \in \cO_K$ since $\dd(T)$ is an element of $K$ that is integral at every finite prime.

\item Compute $\sigma_{\pi}(t) := \sigma_p(t)$ as a formal power series in $t\Zp[[t]]$ with sufficient precision. This equality holds since our elliptic curve $E$ is defined over $\Q$.

\item We use Proposition \ref{packaging} to determine the anticyclotomic $p$-adic height of $T$: compute  
\begin{align*}
h_\rho(T)&= \frac{1}{p}\log_p\left(\psi\left(\frac{\sigma_\pi(T)}{\sigma_\pi(T^c)} \right)\right)+\frac{1}{hp}\log_p\left(\psi\left(\frac{\dd(T )^c}{\dd( T)}\right)\right)\\
&=  \frac{1}{p}\log_p\left(\psi\left(\frac{\sigma_p\left(\frac{-x(T)}{y(T)}\right)}{\sigma_p\left(\frac{-x(T)^c}{y(T)^c}\right)}\right)\right) + \frac{1}{hp}\log_p\left(\psi\left(\frac{\dd(T )^c}{\dd( T)}\right)\right)\\
&= \frac{1}{p}\log_p
\left( \frac {\sigma_p\left(\psi\left(\frac{-x(T)}{y(T)}\right)\right)} {\sigma_p\left(\psi\left(\frac{-x(T)^c}{y(T)^c}\right)\right)}\right) 
+ \frac{1}{hp}\log_p\left(\psi\left(\frac{\dd(T )^c}{\dd( T)}\right)\right).\\
\end{align*}

\item Output the anticyclotomic $p$-adic height of $P$: compute\footnote{As a consistency check we compute the height of $nP$ and verify that $h_{\rho}(nP) = \frac{1}{n^2}h_{\rho}( P)$ for positive integers $n\leq 5$.} 
\[
h_\rho( P) = \frac{1}{n^2m_0^2}h_\rho(T).
\]

\end{enumerate}
\end{algorithm}

\bigskip

\begin{algorithm} \label{shadow} Shadow line attached to $(E,K,p)$.\\
\textit{Input}: 
\begin{itemize}
\item an elliptic curve $E/\Q$ (given by its reduced minimal model) of analytic rank $2$ such that $\mathrm {rank}_\Z E(\Q) =2$;
\item an odd prime $p$ of good ordinary reduction such that the $p$-primary part of the Tate-Shafarevich group of $E/\Q$ is finite; 
\item two $\Z$-independent points $P_1, P_2 \in E(\Q)$; 
\item an imaginary quadratic field $K$ such that 
\begin{itemize}
\item the analytic rank of $E/K$ equals $3$, and 
\item $p$ and all rational primes dividing the conductor of $E/\Q$ split in $K$.    
\end{itemize}
\end{itemize}
\textit{Output}: The slope of the shadow line $L_K \subseteq E(\Q)\otimes\Q_p$ with respect to the coordinate system induced by $\{P_1, P_2\}$. \\

\begin{enumerate}
\item Use Algorithm \ref{Kgenerator} to compute a non-torsion point $S \in E(K)^-$. We then have generators $P_1, P_2, S$ of $E(K)\otimes \Qp$ such that $P_1, P_2 \in E(\Q)$ and $S \in E(K)^-$ (given as points on the reduced minimal model of $E/\Q$) . 

\item Compute $P_1 + S$ and $P_2 +S$.

\item Use Algorithm~\ref{height} to compute\footnote{We compute the height of $P_1 + P_2 + S$  as a consistency check.} the anticyclotomic $p$-adic heights: $h_\rho(P_1+ S)$ and $h_\rho(P_2+S)$. 
Finding that at least one of these heights is non-trivial implies that the shadow line associated to $(E,K,p)$ lies in $E(\Q)\otimes \Q_p$, i.e., the Sign Conjecture holds for $(E, K, p)$.
\item The point $h_\rho(P_2+S) P_1-h_\rho(P_1+S) P_2$ is a generator of the shadow line associated to $(E,K,p)$.  Output the slope of the shadow line $L_K \subseteq E(\Q)\otimes\Q_p$ with respect to the coordinate system induced by $\{P_1, P_2\}$: compute
\[
-h_\rho(P_1+S)/ h_\rho(P_2+S) \in \Q_p.
\]
\end{enumerate}
\end{algorithm}

\bigskip

\section{Examples}\label{Examples}

Let $E$ be the elliptic curve $\mlq\mlq 389.a1\mrq\mrq$ \cite[\href{http://www.lmfdb.org/EllipticCurve/Q/389.a1}{Elliptic Curve 389.a1}]{lmfdb} given by the model 
\[
y^2 + y = x^3 + x^2 - 2x.
\]
We know that the analytic rank of $E/\Q$ equals $2$ \cite[\S6.1]{Bradshaw} and $\rank_\Z E(\Q) = 2$, see \cite{Cremona}. In addition, $5$ and $7$ are good ordinary primes for $E$.  We find two $\Z$-independent points
\[
P_1 = (-1,1), P_2 = (0,0) \in E(\Q).
\] 
We will now use the algorithms described in \S\ref{Algs} to compute the slopes of two shadow lines on $E(\Q)\otimes \Q_5$ with respect to the coordinate system induced by $\{P_1, P_2\}$.

\bigskip
\subsection{Shadow line attached to $\big(\mlq\mlq 389.a1\mrq\mrq, \Q(\sqrt{-11}),5\big)$}\label{ex1}

The imaginary quadratic field $K = \Q(\sqrt{-11})$ satisfies the Heegner hypothesis for $E$ and the quadratic twist $E^K$ has analytic rank $1$.  Moreover, the prime $5$  splits in $K$.

We use Algorithm \ref{Kgenerator} to find a non-torsion point $S = (\frac{1}{4} , \frac{1}{8}\sqrt{-11} - \frac{1}{2}) \in E(K)^-$. We now proceed to compute the anticyclotomic $p$-adic heights of $P_1 + S$ and $P_2 + S$ which are needed to determine the slope of the shadow line associated to the triple $\big(\mlq\mlq389.a1\mrq\mrq, \Q(\sqrt{-11}),5\big)$.
We begin by computing 
\begin{align*}
A_1 &:= P_1 + S = \left(-\frac{6}{25}\sqrt{-11} + \frac{27}{25}, -\frac{62}{125}\sqrt{-11} + \frac{29}{125}\right),\\
A_2 &:= P_2 + S = (-2\sqrt{-11}, -4\sqrt{-11}- 12).
\end{align*}

We carry out the steps of Algorithm \ref{height} to compute $h_{\rho}(A_1)$:
\begin{enumerate}
\item Let $5 \cO_K = \pi \pi^c$, where $\pi = (\frac{1}{2}\sqrt{-11} + \frac{3}{2})$ and  $\pi^c = (-\frac{1}{2}\sqrt{-11} + \frac{3}{2})$. This allows us to fix an identification  
\[
\psi: K_{\pi} \rightarrow \Q_5
\]
that sends 
\[
\frac{1}{2}\sqrt{-11} + \frac{3}{2} \mapsto 2 \cdot 5 + 5^{2} + 3 \cdot 5^{3} + 4 \cdot 5^{4} + 4 \cdot 5^{5} + 3 \cdot 5^{7} + 5^{8} + 5^{9} + O(5^{10}).
\]
\item  Since the Tamagawa number at 389 is trivial, i.e., $c_{389} = 1$, we have $m_0 = 1$. Thus $R = A_1$. 
\item  We find that $n = 9$ is the smallest multiple of $R$ and $R^c$ such that both points reduce to $0$ in $E(\cO_K/\pi)$. Set $T = 9R$.
\item Note that the class number of $K$ is $h = 1$. We find $\dd( R) = \frac{1}{2} \sqrt{-11} - \frac{3}{2}$.
\item Let $f_9$ denote the $9$th division polynomial associated to $E$. We compute 
\begin{align*}
\dd(T) &= \dd(9R) \\
&= f_9(R ) \dd(R )^{9^2}\\
&= 24227041862247516754088925710922259344570 \sqrt{-11} \\
&\qquad - 147355399895912034115896942557395263175125.
\end{align*}

\item We compute 
\begin{align*}\sigma_{\pi}(t) :&= \sigma_5(t)\\
&=t + \left(4 + 5 + 3 \cdot 5^{2} + 5^{3} + 2 \cdot 5^{4} + 3 \cdot 5^{5} + 2 \cdot 5^{6} + O(5^{8})\right)t^{3} \\
&\quad+ \left(3 + 2 \cdot 5 + 2 \cdot 5^{2} + 2 \cdot 5^{3} + 2 \cdot 5^{4} + 2 \cdot 5^{5} + 2 \cdot 5^{6} + O(5^{7})\right)t^{4} \\
&\quad+ \left(1 + 5 + 5^{2} + 5^{3} + 3 \cdot 5^{4} + 3 \cdot 5^{5} + O(5^{6})\right)t^{5}\\
&\quad + \left(4 + 2 \cdot 5 + 2 \cdot 5^{2} + 2 \cdot 5^{3} + 3 \cdot 5^{4} + O(5^{5})\right)t^{6}\\
&\quad + \left(4 + 3 \cdot 5 + 4 \cdot 5^{2} + O(5^{4})\right)t^{7} + \left(3 + 3 \cdot 5^{2} + O(5^{3})\right)t^{8}\\
&\quad + \left(3 \cdot 5 + O(5^{2})\right)t^{9} + \left(2 + O(5)\right)t^{10} + O(t^{11}).
\end{align*} 

\item We use Proposition \ref{packaging} to determine the anticyclotomic $p$-adic height of $T$: we compute  
\begin{align*}
h_\rho(T)&= \frac{1}{p}\log_p\left(\psi\left(\frac{\sigma_\pi(T)}{\sigma_\pi(T^c)} \right)\right)+\frac{1}{hp}\log_p\left(\psi\left(\frac{\dd(T )^c}{\dd( T)}\right)\right)\\
&= \frac{1}{p}\log_p
\left( \frac {\sigma_p\left(\psi\left(\frac{-x(T)}{y(T)}\right)\right)} {\sigma_p\left(\psi\left(\frac{-x(T)^c}{y(T)^c}\right)\right)}\right) 
+ \frac{1}{hp}\log_p\left(\psi\left(\frac{\dd(T )^c}{\dd( T)}\right)\right)\\
&= 3 + 5 + 5^{2} + 4 \cdot 5^{4} + 3 \cdot 5^{5} + 4 \cdot 5^{7} + 3 \cdot 5^{8} + 5^{9} +O(5^{10}).
\end{align*}

\item We output the anticyclotomic $p$-adic height of $A_1$: 
\begin{align*}
h_\rho(A_1) &= \frac{1}{9^2}h_\rho(T) \\
&=  3 + 3 \cdot 5 + 3 \cdot 5^{2} + 2 \cdot 5^{4} + 4 \cdot 5^{5} + 4 \cdot 5^{6} + 3 \cdot 5^{8} + O(5^{10}).
\end{align*}
\end{enumerate}

Repeating Steps (1) -- (8) for $A_2$ yields
\[
h_{\rho}(A_2) = 3 + 2 \cdot 5 + 4 \cdot 5^{2} + 2 \cdot 5^{5} + 5^{6} + 4 \cdot 5^{7} + 4 \cdot 5^{9} +O(5^{10}).
\]

As a consistency check, we also compute 
\[
h_\rho(P_1 + P_2 + S) = 1 + 5 + 3 \cdot 5^{2} + 5^{3} + 2 \cdot 5^{4} + 5^{5} + 5^{6} + 4 \cdot 5^{8} + 4 \cdot 5^{9} + O(5^{10}).
\]
Observe that, numerically, we have 
\[
h_\rho(P_1 + P_2 + S)  = h_\rho(P_1 + S) + h_\rho(P_2 + S).
\]

The slope of the shadow line $L_K \subseteq E(\Q)\otimes\Q_p$ with respect to the coordinate system induced by $\{P_1, P_2\}$ is thus
\[
-\frac{h_\rho(P_1+S)}{h_\rho(P_2+S)}=4 + 2 \cdot 5 + 5^{2} + 3 \cdot 5^{3} + 5^{4} + 5^{6} + 5^{7} + O(5^{10}).
\]

\bigskip

\subsection{Shadow line attached to $\big(\mlq\mlq 389.a1\mrq\mrq, \Q(\sqrt{-24}),5\big)$}

Consider the imaginary quadratic field $K = \Q(\sqrt{-24})$. Note that $K$ satisfies the Heegner hypothesis for $E$, the twist $E^K$ has analytic rank $1$, and the prime $5$ splits in $K$.

Using Algorithm \ref{Kgenerator} we find a non-torsion point $S = \left(\frac{1}{2} ,\frac{1}{8} \sqrt{-24} - \frac{1}{2}\right)  \in E(K)^-$. We then compute
\begin{align*}P_1 + S &= \left(-\frac{1}{6} \sqrt{-24} + \frac{1}{3}, -\frac{5}{18} \sqrt{-24} - 1\right)\\
P_2 + S &= \left(-\frac{1}{2} \sqrt{-24} - 2, -6\right).\end{align*}

Many of the steps taken to compute $h_\rho(P_1 + S)$ and $h_\rho(P_2 + S)$ are quite similar to those in $\S\ref{ex1}$. One notable difference is that in this example the class number $h$ of $K$ is equal to $2$. We find that 
\begin{align*}
h_\rho(P_1 + S) &= 4 + 2 \cdot 5 + 3 \cdot 5^{4} + 2 \cdot 5^{5} + 4 \cdot 5^{6} + 2 \cdot 5^{7} + 5^{8} + 2 \cdot 5^{9} + O(5^{10}),\\
h_\rho(P_2 + S ) &=1 + 5 + 5^{3} + 5^{5} + 2 \cdot 5^{6} + 4 \cdot 5^{7} + 2 \cdot 5^{8} + 3 \cdot 5^{9} + O(5^{10}).\\
\end{align*}
In addition, we compute $h_\rho(P_1 + P_2 + S)$ and verify that
\begin{align*}
h_\rho(P_1 + P_2 + S) &= 4 \cdot 5 + 5^{3} + 3 \cdot 5^{4} + 3 \cdot 5^{5} + 5^{6} + 2 \cdot 5^{7} + 4 \cdot 5^{8} + O(5^{10})\\
&= h_\rho(P_1 +S) + h_\rho(P_2 + S).
\end{align*}
This gives that the slope of the shadow line $L_K \subseteq E(\Q)\otimes\Q_p$ with respect to the coordinate system induced by $\{P_1, P_2\}$ is
\[
-\frac{h_\rho(P_1+S)}{h_\rho(P_2+S)}= 1 + 5 + 3 \cdot 5^{2} + 3 \cdot 5^{5} + 3 \cdot 5^{6} + 3 \cdot 5^{7} + 2 \cdot 5^{8} + 5^{9} +O(5^{10}).
\]

\subsection{Summary of results of additional computations of shadow lines}

The algorithms developed in \S \ref{Algs} enable us to compute shadow lines in many examples which is what is needed to initiate a study of Question \ref{Quest}. We will now list some results of additional computations of slopes of shadow lines on the elliptic curve $\mlq\mlq 389.a1\mrq\mrq$. In the following two tables we fix the prime $p=5, 7$ respectively, and vary the quadratic field.

\smallskip
\begin{table}[h]
\begin{center}
\caption{Slopes of shadow lines for $\big(\mlq\mlq 389.a1\mrq\mrq,K, 5\big)$}
 \begin{tabular}{||c |  l||}
    \hline
 $K$  & slope  \\ \hline
 $\Q(\sqrt{-11})$ & $4 + 2 \cdot 5 + 5^{2} + 3 \cdot 5^{3} + 5^{4} + 5^{6} + 5^{7} + O(5^{10})$ \\
 $\Q(\sqrt{-19})$ & $1 + 4 \cdot 5 + 2 \cdot 5^{2} + 2 \cdot 5^{3} + 2 \cdot 5^{5} + 5^{6} + 4 \cdot 5^{7} + 3 \cdot 5^{8} + 4 \cdot 5^{9} + O(5^{10})$ \\
 $\Q(\sqrt{-24})$ & $1 + 5 + 3 \cdot 5^{2} + 3 \cdot 5^{5} + 3 \cdot 5^{6} + 3 \cdot 5^{7} + 2 \cdot 5^{8} + 5^{9} + O(5^{10})$ \\
 $\Q(\sqrt{-59})$ & $4 + 5 + 4 \cdot 5^{2} + 5^{3} + 2 \cdot 5^{4} + 2 \cdot 5^{5} + 3 \cdot 5^{7} + 4 \cdot 5^{8} + 2 \cdot 5^{9} + O(5^{10})$\\
 $\Q(\sqrt{ -79})$ & $2 + 5 + 2 \cdot 5^{2} + 2 \cdot 5^{3} + 4 \cdot 5^{4} + 4 \cdot 5^{5} + 3 \cdot 5^{6} + 3 \cdot 5^{7} + 3 \cdot 5^{8} + 2 \cdot 5^{9} + O(5^{10})$\\
 $\Q(\sqrt{ -91})$ & $4 + 3 \cdot 5 + 5^{2} + 5^{4} + 2 \cdot 5^{5} + 4 \cdot 5^{6} + 5^{7} + 2 \cdot 5^{9} + O(5^{10})$\\
 $\Q(\sqrt{ -111})$ & $5^{-2} + 4 \cdot 5^{-1} + 4 + 4 \cdot 5 + 2 \cdot 5^{2} + 2 \cdot 5^{3} + 4 \cdot 5^{4} + 2 \cdot 5^{5} + 3 \cdot 5^{6} + 5^{7} + 2 \cdot 5^{8} + 5^{9} + O(5^{10})$\\
 $\Q(\sqrt{ -119})$ & $4 \cdot 5^{-1} + 2 + 2 \cdot 5 + 2 \cdot 5^{2} + 4 \cdot 5^{3} + 4 \cdot 5^{4} + 2 \cdot 5^{5} + 5^{6} + 4 \cdot 5^{7} + 4 \cdot 5^{8} + 4 \cdot 5^{9} + O(5^{10})$\\
 $\Q(\sqrt{ -159})$ & $2 \cdot 5 + 4 \cdot 5^{4} + 4 \cdot 5^{5} + 5^{6} + 5^{7} + 4 \cdot 5^{8} + 5^{9} + O(5^{10})$\\
 $\Q(\sqrt{ -164})$ & $3 + 2 \cdot 5 + 4 \cdot 5^{2} + 5^{3} + 4 \cdot 5^{4} + 3 \cdot 5^{5} + 3 \cdot 5^{6} + 3 \cdot 5^{8} + 4 \cdot 5^{9} + O(5^{10})$\\
 \hline
 \end{tabular}
 \end{center}
 \end{table}
 
\smallskip
\begin{table}[h]
\begin{center}
\caption{Slopes of shadow lines for $\big(\mlq\mlq 389.a1\mrq\mrq,K, 7\big)$}

 \begin{tabular}{||c |  l||}
    \hline
 $K$  & slope  \\ \hline
 $\Q(\sqrt{-19})$ & $3 + 2 \cdot 7 + 2 \cdot 7^{2} + 3 \cdot 7^{3} + 7^{4} + 7^{5} + 4 \cdot 7^{7} + 6 \cdot 7^{9} + O(7^{10})$\\
 $\Q(\sqrt{ -20})$ & $1 + 5 \cdot 7 + 6 \cdot 7^{2} + 6 \cdot 7^{3} + 2 \cdot 7^{4} + 3 \cdot 7^{5} + 3 \cdot 7^{6} + 3 \cdot 7^{7} + O(7^{10})$ \\
 $\Q(\sqrt{ -24})$ & $1 + 3 \cdot 7 + 3 \cdot 7^{2} + 2 \cdot 7^{3} + 6 \cdot 7^{4} + 2 \cdot 7^{5} + 2 \cdot 7^{6} + 6 \cdot 7^{7} + 2 \cdot 7^{8} + O(7^{10})$\\
 $\Q(\sqrt{ -52})$ & $1 + 5 \cdot 7 + 7^{2} + 3 \cdot 7^{3} + 3 \cdot 7^{4} + 2 \cdot 7^{5} + 5 \cdot 7^{6} + 3 \cdot 7^{9} + O(7^{10})$\\
 $\Q(\sqrt{ -55})$ & $1 + 7 + 6 \cdot 7^{2} + 3 \cdot 7^{3} + 5 \cdot 7^{4} + 3 \cdot 7^{5} + 7^{7} + 4 \cdot 7^{9} + O(7^{10})$ \\
 $\Q(\sqrt{ -59})$ & $2 + 7 + 3 \cdot 7^{2} + 3 \cdot 7^{3} + 5 \cdot 7^{4} + 5 \cdot 7^{5} + 2 \cdot 7^{6} + 4 \cdot 7^{7} + 7^{8} + 6 \cdot 7^{9} + O(7^{10})$ \\
 $\Q(\sqrt{ -68})$ & $4 + 4 \cdot 7 + 2 \cdot 7^{3} + 5 \cdot 7^{4} + 5 \cdot 7^{5} + 7^{6} + 7^{7} + 5 \cdot 7^{8} + 5 \cdot 7^{9} + O(7^{10})$ \\
 $\Q(\sqrt{ -87})$ & $3 \cdot 7 + 4 \cdot 7^{2} + 7^{3} + 2 \cdot 7^{4} + 2 \cdot 7^{5} + 7^{6} + 5 \cdot 7^{7} + 7^{9} + O(7^{10})$ \\
 $\Q(\sqrt{ -111})$ & $7^{-2} + 2 \cdot 7^{-1} + 5 + 2 \cdot 7 + 7^{2} + 2 \cdot 7^{3} + 6 \cdot 7^{4} + 5 \cdot 7^{5} + 7^{6} + 2 \cdot 7^{7} + 2 \cdot 7^{9} + O(7^{10})$\\
 $\Q(\sqrt{ -143})$ & $5 + 5 \cdot 7 + 2 \cdot 7^{3} + 4 \cdot 7^{4} + 3 \cdot 7^{5} + 3 \cdot 7^{6} + 2 \cdot 7^{7} + 2 \cdot 7^{9} + O(7^{10})$ \\

 \hline 
 \end{tabular}
 \end{center}
 \end{table}

\end{document}